\documentclass[reqno,11pt]{amsart}

\usepackage{tikz, pgfplots,epsf, mathtools, subcaption}
\usepackage{amssymb,graphics,graphicx,wrapfig}
\usepackage{float}
\usepackage{dsfont}
\usetikzlibrary{decorations.pathreplacing}
\usetikzlibrary{patterns}
\usetikzlibrary{fit}

\def\eps{\varepsilon}

\def\N{\bbN}

\def\be{\begin{equation}}
	\def\ee{\end{equation}}
		\def\bfig{\begin{figure}[htb]}
			\def\efig{\end{figure}}
		
		\numberwithin{equation}{section}
		\newtheorem{theorem}{Theorem}[section]
		
		\newtheorem{proposition}[theorem]{Proposition}
		\newtheorem{lemma}[theorem]{Lemma}

		\newcommand{\comment}[1]{}

		\DeclareMathSymbol{\leqslant}{\mathalpha}{AMSa}{"36}
		\DeclareMathSymbol{\geqslant}{\mathalpha}{AMSa}{"3E}
		\DeclareMathSymbol{\doteqdot}{\mathalpha}{AMSa}{"2B}
		\DeclareMathSymbol{\circlearrowright}{\mathalpha}{AMSa}{"08}
		\DeclareMathSymbol{\subsetneq}{\mathalpha}{AMSb}{"28}
		\DeclareMathSymbol{\supsetneq}{\mathalpha}{AMSb}{"29}
		\renewcommand{\leq}{\;\leqslant\;}
		\renewcommand{\geq}{\;\geqslant\;}

		\newcommand{\e}[1]{\,{\rm e}^{#1}\,}

		\newcommand{\upchi}{\raise 2pt \hbox{$\chi$}}

		\def\writefig#1 #2 #3 {\rlap{\kern #1 truecm \raise #2 truecm
				\hbox{#3}}}
		
		\newcommand{\caC}{{\mathcal C}}

		\newcommand{\caL}{{\mathcal L}}

		\def\bbone{{\mathchoice {\rm 1\mskip-4mu l} {\rm 1\mskip-4mu l} {\rm 1\mskip-4.5mu l} {\rm 1\mskip-5mu l}}}

		\newcommand{\bbE}{{\mathbb E}}

		\newcommand{\bbN}{{\mathbb N}}
		
		\newcommand{\bbP}{{\mathbb P}}
		
		\newcommand{\bbR}{{\mathbb R}}

		\newcommand{\bse}{{\boldsymbol e}}

		\newcommand{\R}{\mathbb{R}} 
		
		\renewcommand{\epsilon}{\varepsilon}
		
		\newcommand{\E}{\mathbb{E}}
		\renewcommand{\P}{\mathbb{P}}

		\renewcommand{\L}{\mathcal{L}}

		\newcommand{\numb}[1]{\mathrm{N}_{#1}}

		\newcommand{\uproman}[1]{\uppercase\expandafter{\romannumeral#1}}

		\usepackage{listings}
		\usepackage{xcolor}
		\lstdefinestyle{pseudo}{
			backgroundcolor=\color{white},   
			basicstyle=\ttfamily,            
			frame=leftline,                  
			rulecolor=\color{black},         
			numbers=none,                    
			keywordstyle=\bfseries\color{blue}, 
			commentstyle=\itshape\color{green}, 
			morekeywords={If, Else, if, For, all, Do, While, Return, Skip}, 
			xleftmargin=0.5em,                 
			breaklines=true,                 
			literate={:}{{\textcolor{blue}{:}}}1, 
			showspaces=false,
			showstringspaces=false,
			tabsize=1,                      
			escapeinside={(*@}{@*)}, 
		}

		\title{Loop percolation versus link percolation in the random loop model}
		\author{Volker Betz}
		\address{Volker Betz \hfill\newline
			\indent FB Mathematik, TU Darmstadt}
		\email{betz@mathematik.tu-darmstadt.de}
		
		\author{Andreas Klippel}
		\address{Andreas Klippel \hfill\newline
			\indent FB Mathematik, TU Darmstadt}
		\email{andreas.klippel@math.tu-darmstadt.de}
		
		\author{Mino Nicola Kraft}
		\address{Mino Nicola Kraft \hfill\newline
			\indent FB Mathematik, TU Darmstadt}
		\email{mino\_nicola.kraft@tu-darmstadt.de}
		
		\date{\today}

\begin{document}
			
			\maketitle
			
			\begin{abstract}
				In \cite{Muelbacher2021}, Peter M\"uhlbacher showed that in the random loop model without loop weights, 
				a loop phase transition (assuming it exists) cannot occur at the same parameter as the percolation 
				phase transition of the occupied edges. In this work, we give a quantitative version of this result, specifying a minimal 
				gap between the percolation phase transition and a possible loop phase transition. A substantial part of our argument also works
				for weighted loop models. 		
			\end{abstract}

			\section{Introduction and main results}\label{introduction}
			Random loop models arise as graphical representations of various quantum spin systems, such as the quantum Heisenberg 
			(anti-) ferromagnet. These connections were first observed in \cite{Aizenman1994,Tth1993} and then extended in \cite{Ueltschi2013}. 
			We refer the reader to these works for the connections to quantum systems and will exclusively treat the probabilistic versions here. 
			
			Given a finite graph, $ G=( V,E)$, a {\em link configuration} is a finite sequence $c = (c_i)_{i \in[n]} = 
			(e_i, s_i)_{i \in[n]}$ with $n \in \mathbb N$, $e_i \in E$ and $s_i \in \{-1,1\}$ for all $i \in[n]$\footnote{We use the notation $[n]\coloneqq\{1,\dots,n\}$ and $[n]_0\coloneqq [n]\cup\{0\}$ for $n\in\N$ throughout.}. 
			$s_i = 1$ corresponds to a 'cross' and $s_i = -1$ to a 'double bar' as described in  \cite{Ueltschi2013}. One element of the sequence is called a {\em link}. 
			
			A link configuration $c = (e_i, s_i)_{i \in[n]}$ gives rise to a loop configuration. The most common way to describe this is 
			through a graphical construction, see e.g.\  \cite{Ueltschi2013}. Here we provide an alternative, equivalent definition. 
			We define the loop configuration $\caL(c)$ induced by the link configuration $c = (c_i)_{i \in[n]} = (e_i, s_i)_{i \in[n]}$ 
			as a partition of the
			set $X = V \times [n]$. This partition is given by the equivalence classes arising from an equivalence relation $\sim$ which we define as follows: 
			for $(v,j), (w,k) \in X$, we say that $(v,j) \sim (w,k)$ if $v=w$ and $j=k$, or if one of the following conditions is satisfied:
			\[
			\begin{split} 
			(i) & \qquad v = w, k = (j+1) \text{ mod } n, \text{ and } v \notin e_j, \\
			(ii) & \qquad \{v,w\} \in E, j=k, e_j = \{v,w\}, \text{ and } s_j = -1, \\
			(iii) & \qquad \{v,w\} \in E, j=k, e_{(j-1) \text{ mod } n} = \{v,w\}, \text{ and } s_j = -1, \\
			(iv) & \qquad \{v,w\} \in E, k = (j+1) \text{ mod } n, e_j = \{v,w\}, \text{ and } s_j = 1.
			\end{split}
			\]
			We extend $\sim$ first by symmetry, and then by transitivity, to an equivalence relation.  
			$\caL(c)$ denotes the set of equivalence classes of the resulting equivalence relation. The connection to
			the graphical construction appears when interpreting $(v,j)$ as the interval $((j-1)/n, j/n)$ on a real axis attached to $v$: then (i)
			tells us that intervals on adjacent levels on the same vertex with no intervening links are connected; (ii) and (iii) say 
			that double bars connect intervals of the same level at both vertices of the edge where they appear; 
			(iv) means that intervals of adjacent levels on neighbouring vertices are connected if 
			a cross is placed on their shared edge.  See Figure \ref{fig:LoopConstruction}. 	
			
			\begin{figure}[t]
				\centering		
					\begin{tikzpicture}[xscale=0.8,yscale=0.9]
					\draw[->, black] (1,0) -- (9,0) node[right] {$ G$};
					\draw[->, black] (1,0) -- (1,3) node[above] {$i$};	\foreach \x in {1,2,3,4}{		\draw[black] (2*\x,0) -- (2*\x,-0.3) node[below] {$v_{\x}$};
						\draw[black, line width=1pt] (2*\x,0) -- (2*\x,3);
						
					}
					\foreach \x in {1,2,3}{\draw[black] (2*\x+1,0) -- (2*\x+1,0) node[below]{$e_{\x}$};}
					\draw[red, line width=1pt] (2,0.45*0) -- (2,0.45*6-0.06);
					\draw[red, line width=1pt] (2,0.45*6+0.06) -- (2,3);
					\draw[red, line width=1pt] (4,0) -- (4,0.45*1-0.06);
					\draw[red, line width=1pt] (4,0.45*1+0.06) -- (4,0.45*3-0.06);
					\draw[red, line width=1pt] (4,0.45*4+0.06) -- (4,0.45*6-0.06);
					\draw[red, line width=1pt] (4,0.45*6+0.06) -- (4,3);
					\draw[red, line width=1pt] (6,0) -- (6,0.45*1-0.06);
					\draw[red, line width=1pt] (6,0.45*1+0.06) -- (6,0.45*2-0.06);
					\draw[red, line width=1pt] (6,0.45*2+0.06) -- (6,0.45*3-0.06);
					\draw[red, line width=1pt] (6,0.45*4+0.06) -- (6,0.45*5-0.06);
					\draw[red, line width=1pt] (6,0.45*5+0.06) -- (6,3);
					\draw[red, line width=1pt] (8,0) -- (8,0.45*2-0.06);
					\draw[red, line width=1pt] (8,0.45*2+0.06) -- (8,0.45*5-0.06);
					\draw[red, line width=1pt] (8,0.45*5+0.06) -- (8,3);
					\draw[green, line width=1pt] (4,0.45*3+0.06) -- (4,0.45*4-0.06);
					\draw[green, line width=1pt] (6,0.45*3+0.06) -- (6,0.45*4-0.06);
					\foreach \x in {1,2,3,4,5,6}{		
						\draw[black] (1,0.45*\x) -- (1,0.45*\x) node[left] {$\x$};
					}
					\draw[black] (1,0) -- (1,0) node[left] {$0$};

					\draw[red, line width=1pt] (2,0.45*6-0.06) -- (4,0.45*6+0.06);
					\draw[red, line width=1pt] (2,0.45*6+0.06) -- (4,0.45*6-0.06);
					\draw[red, line width=1pt] (4,0.45*1-0.06) -- (6,0.45*1-0.06);
					\draw[red, line width=1pt] (4,0.45*1+0.06) -- (6,0.45*1+0.06);
					\draw[red, line width=1pt] (4,0.45*3-0.06) -- (6,0.45*3-0.06);
					\draw[green, line width=1pt] (4,0.45*3+0.06) -- (6,0.45*3+0.06);
					\draw[green, line width=1pt] (4,0.45*4-0.06) -- (6,0.45*4-0.06);
					\draw[red, line width=1pt] (4,0.45*4+0.06) -- (6,0.45*4+0.06);
					\draw[red, line width=1pt] (6,0.45*2-0.06) -- (8,0.45*2-0.06);
					\draw[red, line width=1pt] (6,0.45*2+0.06) -- (8,0.45*2+0.06);
					\draw[red, line width=1pt] (6,0.45*5-0.06) -- (8,0.45*5+0.06);
					\draw[red, line width=1pt] (6,0.45*5+0.06) -- (8,0.45*5-0.06);
					\end{tikzpicture}
			\caption{We consider the link configuration $((e_2,-1),(e_3,-1),(e_2,-1),(e_2,-1),(e_3,1),(e_1,1))$. Following the rules of the equivalence 
				relation, we obtain two equivalence classes (loops), coloured in green and red. }
			\label{fig:LoopConstruction}
			\end{figure} 
			
			The random loop measure on the graph $G = (V,E)$ is defined to be the probability measure $\bbP_{\beta,u,\theta}$ on 
			\[
			\caC(E) = \{(c_i)_{i \in[n]} = (e_i, s_i)_{i \in[n]}: n \in \bbN, e_i \in E ,\text{ and } s_i \in \{-1,1\} \text{ for all } i \in[n] \}
			\]
			with  
			\begin{align}\label{probweights1}
				\bbP_{\beta,u,\theta}(\{c\})\coloneqq\left(Z_{\beta,u,\theta}\right)^{-1}\frac{\beta^{|c|}}{|c|!}  
				u^{\frac12 \sum_{j=1}^n (1 + s_j)}(1-u)^{\frac12 \sum_{j=1}^n (1 - s_j) }\theta^{|\L(c)|},
			\end{align} 	
			where $|c|$ is the length of the vector $c \in \mathcal C( E)$ and $Z_{\beta,u,\theta}$ is the normalisation constant.

			A constructive way of thinking about the random loop measure is to imagine that we first draw the length $|c|$ of the 
			link configuration by evaluating a Poisson random variable with parameter $\beta |E|$. Then for each element $c_i$ of the vector, we 
			assign independently
			a first component $e_i$ using the uniform distribution on $E$, and a type $s_i$, choosing $s_i=1$ with probability $u$ and $s_i=-1$ with 
			probability $1-u$. \\
			Then the quantities 
			\[
			N_e(c) = | \{ i \in[|c|]: c_i = (e,s) \text { for some } s \in \{-1,1\} \} |
			\]
			are i.i.d. Poisson random variables with parameter $\beta$, and the order in which the elements appear in the sequence 
			is uniformly distributed on the set of all possible orderings. This 
			gives us the connection with the usual way of defining the model. 
			The resulting sequence has law $\bbP_{\beta,u,1}$. To obtain $\bbP_{\beta,u,\theta}$, a reweighting with the 
			weight function $\theta^{|\caL|}$ is necessary.  For $\theta > 1$, this favours configurations with more loops. 
			For integer $\theta$, there is a combinatorial interpretation of colouring each loop with one of 
			$\theta$ different colours.

			The main question about the random loop model concerns the existence of infinite loops, and is therefore a percolation-type question. 
			For a finite graph $G=(V,E)$ and a link configuration 
			$c \in \mathcal C(E)$, we say that $v,w \in V$ are {\em connected by a  loop}, and write $v \stackrel{c}\Longleftrightarrow w$, 
			if $(v,1)$ and $(w,1)$ are in the same equivalence class.\\ 
			For an infinite connected graph $ G = ( V, E)$, we say that infinite loops exist if we can find an increasing 
			sequence $(E_m)_{m\geq1} \in  E^\N$ and a vertex $v_0 \in  e$ for some $e\in E_1$ 
			such that: \\[1mm]
			(i) $\bigcup_{m \in \mathbb N}  E_m =  E$, and the subgraph $G_m=(V_m,E_m)$ of $G$ generated by $E_m$ is connected
			for all $m$. \\[1mm]
			(ii) Let $\mathbb P_{{\beta,u,\theta,m}}$ denote the random loop measure on $G_m$, and let $d(v,w)$ be the graph distance of two vertices $v$ and $w$. Then
			\begin{equation} \label{eq:loop percolation}
				\lim_{R \to \infty} \liminf_{m \to \infty} \mathbb P_{{\beta,u,\theta,m}}(v_0\overset{c}{\Longleftrightarrow} w \text{ for some } w \in V_m \text{ with } 
				d(v_0,w) \geq R) > 0.
			\end{equation}
		
			While it is usually rather easy to show the absence of infinite loops when $\beta$ is small 
			(see also below), positive results are much harder to obtain. 
			Two special graphs that are relatively well understood, are the complete graph and regular trees. For $\theta=1$ and $u=1$, the random loop model is often referred to as interchange process. On complete graphs, the existence of infinitely long loops for the interchange process follows from \cite{Schramm2005}, where it was shown that the re-scaled loop lengths of macroscopic loops converge weakly to the Poisson-Dirichlet distribution ${\rm PD}(1)$ above the critical threshold $\beta=1/2$. \cite{Bjoernberg2019} extended this to $u\in[0,1)$, yielding convergence to the Poisson-Dirichlet distribution ${\rm PD}(1/2)$. A more detailed analysis of the expected number of loops of a certain size can be found in \cite{Berestycki2015}.\\
			On $d$-regular trees, the existence of infinitely long loops for the interchange process has been shown for $d\geq 5$ in \cite{Angel2003} and sharpness of the phase transition has first been proven for $d\geq 764$ in \cite{Hammond2014} and in \cite{Betz2021}, this was extended to $d\geq 3$. Both works, as well as \cite{BjornbergUeltschi2018b,Bjoernberg2018}, contribute an asymptotic expansion in $d$. Here, \cite{Betz2021} allows for $u\in[0,1]$ and \cite{Bjoernberg2018} allows for $\theta\neq 1$. An interesting result for $d$-regular graphs, of which $d$-regular trees are a subclass, was proven in \cite{Poudevigne2022}. There it was shown for $\theta\in(0,\infty)$ that macroscopic cycles appear almost surely if one draws the graph uniformly at random among all $d$-regular graphs. Another study of randomly drawn graphs was carried out in \cite{BetzEtAl2018} where Galton-Watson trees with certain conditions on the offspring distribution were considered.\\
			Two specific graphs that also have been studied, are the hypercube and the Hamming graph. For the first, the existence of long loops was shown in \cite{Kotecky2016} and for the latter \cite{Adamczak2021} extends the results of \cite{Schramm2005} also allowing for $\theta\neq 1$ but leaving the question of criticality open.\\
			Graphs with more complex geometries are notoriously difficult to treat. An important recent success is \cite{elboim2024}, 
			where the existence of loop percolation (in a slightly different sense) is shown in dimensions 5 or higher for the cubic lattice in case $\theta=1$, $u=1$. Further rigorous results on the existence of long loops and quantitative bounds on connection probabilities were recently obtained in \cite{Betz2025}, using refined versions of reflection positivity.\\
			
			In this work, we do not present any results on regimes where \eqref{eq:loop percolation} is valid. 
			Instead, we contribute to the understanding of the 
			region where \eqref{eq:loop percolation} does not hold by comparing the existence of infinite loops and infinite clusters in a percolation model that is called link percolation. 
			For given $c = (e_i, s_i)_{i \in[n]} \in \caC(E)$, we define $\eta(c) \in \{0,1\}^E$ by setting $\eta_e(c) = 1$ if and only if there exists 
			$i \in[n]$ with $e_i =e$. In words, edges are open if at least one link is placed on them.
			The law of $c \mapsto \eta(c)$ under $\bbP_{\beta,u,1}$  is just Bernoulli percolation with probability $1 - \e{-\beta}$ for an open edge, 
			while under the general measure
			$\bbP_{\beta,u,\theta}$ a model of dependent percolation emerges. 
			We write $v \stackrel{c}{\longleftrightarrow} w$ if $v$ and $w$ are in the
			same $\eta(c)$-percolation cluster. As above, we say that there are infinite clusters in link percolation if we can find  
			an increasing 
			sequence $(E_m)_{m\geq1} \in  E^{\N}$ and a vertex $v_0 \in e$ for some $e\in E_1$ such that:\\[1mm]
			(i) $\bigcup_{m \in \mathbb N}  E_m =  E$, and the subgraph $G_m=(V_m,E_m)$ of $G$ generated by $E_m$ is connected
			for all $m$. \\[1mm]
			(ii) Let $\mathbb P_{\beta,u,\theta,m}$ denote the random loop measure on $G_m$, and let $d(v,w)$ be the graph distance of vertices $v$ and $w$. Then, 
			\begin{equation} \label{eq:link percolation}
				\lim_{R \to \infty} \liminf_{m \to \infty} \mathbb P_{{\beta,u,\theta,m}}(v_0 \overset{c}{\longleftrightarrow} w \text{ for some } w \in  V_m \text{ with } 
				d(v_0,w) \geq R) > 0.
			\end{equation}
			It is clear that when \eqref{eq:link percolation} fails to hold, then \eqref{eq:loop percolation} also does not hold. In \cite{Muelbacher2021},
			Peter M\"uhlbacher showed that for graphs of bounded degree, there exists an open interval of parameters where \eqref{eq:link percolation} holds
			but \eqref{eq:loop percolation} does not. 
			The purpose of this article is to give a quantitative version of M\"uhlbacher's result, and at the same time to streamline the 
			proof in several ways. A related result on trees was recently established in \cite{Klippel2025}, where a strict inequality of critical parameters for loop and link percolation is proved for a broad class of random trees including Galton–Watson trees.

			The main idea for comparing link percolation to loop percolation (i.e., existence of infinite loops) is to find sufficiently many edges
			that contribute to the former but not to the latter. For $c \in \caC(E)$, let $N_e(c)$ be as defined above; that is, $N_e(c)$ denotes the number of times the edge $e$ appears in the sequence $(e_i, s_i)_{i\in[n]}$.  
			We say that an edge $e$ is {\em blocking} for a link configuration $c = (e_i, s_i)_{i \in[n]}$ if 
			\\[2mm]
			\begin{minipage}{0.6 \textwidth}
				\begin{enumerate} 
					\item $N_e(c)= 2$,
					\item $s_i = 1$ if $e_i = e$, i.e., both links on $e$ are crosses, and
					\item no adjacent edges carry any links that go between the two links on $e$, i.e.,  if $e_i = e_j = e$ with $i < j$, 
					then $\bigcup_{i < k < j} e_k \cap e = \emptyset$. 
				\end{enumerate}
			\end{minipage}
			\begin{minipage}{0.38 \textwidth}
				\begin{figure}[H]
					\centering		
					\begin{tikzpicture}[xscale=0.5,yscale=0.9]
						\draw[->, black] (1,0) -- (9,0) node[right] {$ G$};
						\draw[->, black] (1,0) -- (1,3) node[above] {$i$};	\foreach \x in {1,2,3,4}{		\draw[black] (2*\x,0) -- (2*\x,-0.3) node[below] {$v_{\x}$};
							\draw[black, line width=1.5pt] (2*\x,0) -- (2*\x,1.2);
							\draw[black, line width=1.5pt] (2*\x,1.88) -- (2*\x,3);
							
						}
						\draw[black, line width=1.5pt] (2,0) -- (2,3);
						\draw[black, line width=1.5pt] (8,0) -- (8,3);
						\foreach \x in {2,4,6,8}{		
							\draw[black] (1,0.3*\x) -- (1,0.3*\x) node[left] {$\x$};
						}
						\draw[black] (1,0) -- (1,0) node[left] {$0$};

						\draw[green, line width=1.5pt] (4,1.2) -- (6,1.28);
						\draw[red, line width=1.5pt] (4,1.28) -- (6,1.2);
						\draw[green, line width=1.5pt] (4,1.88) -- (6,1.8);
						\draw[red, line width=1.5pt] (4,1.8) -- (6,1.88);
						\draw[red, line width=1.5pt] (4,1.28) -- (4,1.8);
						\draw[red, line width=1.5pt] (6,1.2) -- (6,1.08);
						\draw[red, line width=1.5pt] (6,1.88) -- (6,2);
						\draw[green, line width=1.5pt] (4,1.2) -- (4,1.08);
						\draw[green, line width=1.5pt] (4,1.88) -- (4,2);
						\draw[green, line width=1.5pt] (6,1.28) -- (6,1.8);
						\fill[fill=gray!20] (2.02,0) rectangle (3.98,1.08);
						\fill[fill=gray!20] (6.02,0) rectangle (7.98,1.08);
						\fill[fill=gray!20] (2.02,2) rectangle (3.98,3);
						\fill[fill=gray!20] (6.02,2) rectangle (7.98,3);
					\end{tikzpicture}
				\end{figure}
				
			\end{minipage} 
			
			In the illustration
			above, the edge $\{v_2,v_3\}$ is blocking if and only if the fifth element of the link configuration $c$ is not on one of the adjacent edges. 
			If $e = \{v,w\}$ is a blocking edge, a loop that uses one of the links on $e$ to travel from $v$ to $w$, or vice versa, 
			is not diverted from its current position (because of (3)) before it reaches the other link. Therefore, both links can be deleted from 
			$c$ without affecting whether $x \stackrel{c}\Longleftrightarrow y$ for any vertices $x,y \in V$. So, if we define $B_e(c) = 1$
			if $e$ is blocking for $c$, and $B_e(c) = 0$ otherwise, then a sufficient condition for the absence of infinite loops 
			is that the dependent percolation model 
			\[
			c \mapsto (\eta_e(c) (1 - B_e(c)))_{e \in E}
			\]
			has no infinite clusters (again, in the sense of exhausting the graph $G$ with finite approximations). This means that our aim is to find 
			a regime of parameters where the (in general, dependent) percolation model $(\eta_e)_{e \in E}$ has infinite clusters, but 
			$(\eta_e (1 - B_e))_{e \in E}$ has none. Our tool to do this is stochastic domination. We equip the space  $\Omega_E=\{0,1\}^ E$ with 
			the natural partial order $\leq$ given by the entry-wise comparison. 
			A function $f:\Omega_E \to \bbR$ is increasing if $\omega \leq \omega'$ implies $f(\omega) \leq f(\omega')$, and $A \subset \Omega_E$ is
			increasing if its indicator function is an increasing function. For two probability measures $\mu,\nu$ on $\Omega_E$, 
			$\mu$ stochastically dominates $\nu$ if for every increasing function $f$, we have $\nu(f) \leq \mu(f)$. 
									
			Our main result is for general $\theta > 1$. Write 
			\begin{align}
				\hat \theta = \max \{\theta,~1/\theta\},\label{eq:thetahatcheck}
			\end{align} 
			and let $O(c)$ be the set of open edges with respect to $\eta(c)$, i.e., the set of $e \in E$ where $\eta_e(c) = 1$.
			
			\begin{theorem}[Appearance of blocking events]\label{appearanceblocking}
				Let $G=( V, E)$ be a finite graph with bounded edge degree $K$. Define $\bbP_{\beta,u,\theta}$ as in \eqref{probweights1}. 
				Then for any $E_0 \subset E$, the law of $(B_e)_{e \in E_0}$ under 
				$\bbP_{\beta,u,\theta}(\cdot|~O=E_0)$ stochastically dominates a Bernoulli
				edge percolation measure on $E_0$ with parameter 
				\[
				\delta(\beta,u,\theta)\coloneqq\left(\frac{\hat \theta (8K-4)(2K-1)}{\beta u^2}+1\right)^{-1}\e{-\beta^+(2K-2)}.
				\]
			\end{theorem} 
			The idea of using blocking edges for comparison to percolation is taken from \cite{Muelbacher2021}. What is new is the extension to $\theta \neq 1$ and the quantitative bound. For example, we consider the three dimensional lattice ($K=6$), only crosses ($u=1$), $\theta=2$ and $\beta=0.25$\footnote{We use that simulations (see \cite{Wang2013}) indicate the critical threshold to be $\approx 0.25$.}. For these parameters, we get $\delta(\beta,u,\theta)\approx 2.12\cdot10^{-5}$.\\	
			
Theorem \ref{appearanceblocking} allows us to take a percolation cluster of the link percolation model $(\eta_e)_{e \in E}$ and to remove each edge of this 
cluster independently with probability $\delta(\beta,u,\theta)$. If we can find a regime of parameters where the link percolation cluster $\eta$ percolates,
but the 'thinned out' cluster no longer does, we know that for these parameters link percolation occurs, but loop percolation does not. Unfortunately, this
argument requires some control on the law of $(\eta_e)_{e \in E}$ itself near the percolation threshold. The only case where we currently have this control
is the 'trivial' one where $\theta = 1$ and thus $\eta$ is just Bernoulli bond percolation with parameter $1 - e^{-\beta}$. Let 
$p_{\rm c}(G)$ be the critical probability for the graph $G$ with respect to Bernoulli bond percolation. 

\begin{theorem}[Comparison to percolation]\label{comparisonimproved}
				Let $ G$ be an infinite graph with maximal degree $K \in \bbN$. Assume that $\theta = 
				1$, and define the existence of loop percolation and link percolation as in \eqref{eq:loop percolation} 
				and \eqref{eq:link percolation}, 
				respectively. Then for all $\beta > 0$, such that 
				\[
				p_{\rm c}(G) < 1 - \e{-\beta} < \frac{p_{\rm c}( G)}{1 - \delta(\beta,u,1)},
				\]
				loop percolation does not occur, but link percolation does. 
			\end{theorem}

\begin{proof}
Since $(\eta_e)_{e \in E}$ is Bernoulli bond percolation with parameter $1 - e^{-\beta}$, 
link percolation occurs for $\beta$ with $p_{\rm c}(G) < 1 - \e{-\beta}$ by the definition of $p_{\rm c}(G)$.
For the proof that loop percolation does not occur when 
$1 - \e{-\beta} < \frac{p_{\rm c}( G)}{1 - \delta(\beta,u,1)}$, we use Theorem \ref{appearanceblocking}. 

As discussed in the paragraph before Theorem \ref{appearanceblocking}, for a link configuration $c$, 
$v \stackrel{c}\Longleftrightarrow w$ is only possible if 
an open path from $v$ to $w$ exists in the edge percolation configuration $(\eta_e(c) (1 - B_e(c)))_{e\in E}$. Our task is therefore to show that this percolation model
does not exhibit infinite clusters. To do so, we dominate it by a Bernoulli percolation model: let $G = (V,E)$ be a finite graph and 
$f: \{0,1\}^E \to \bbR^+$ be an increasing function. For $E_0 \subset E$ and $\xi \in \{0,1\}^E$, let $\xi|_{E_0}(e) = \xi(e) \bbone_{\{e \in E_0\}}$, 
and define $f_{E_0}(\xi) = f(\xi|_{E_0})$. Note that $f_{E_0}$ is increasing as well.
We abbreviate 
$f(\eta) = f((\eta_e)_{e \in E})$ and analogously $f(\eta (1-B))$. Recall that $O(c)$ is the set of open edges with respect to $\eta(c)$, let $c$ 
be distributed according to $\bbP_{\beta,u,\theta}$ (with $\theta$ general for now), which also fixes the distributions of $\eta$ and $B$. Let $X$ be a Bernoulli 
edge percolation with probability $1 - \delta(\beta, u, \theta)$ for an edge to be open. We write $\bbE$ for the expectation with respect to the 
product measure of $\bbP_{\beta,u,\theta}$ and the probability measure that is associated with $X$, and get 
\[
\begin{split}
\bbE(f(\eta (1 - B))) & = \sum_{E_0 \subset E} \bbE (f(\eta(1-B)) \, | \, O = E_0) \bbP(O = E_0)  \\
& = \sum_{E_0 \subset E} \bbE (f_{E_0}(1-B) \, | \, O = E_0) \bbP(O = E_0) \\
& \leq \sum_{E_0 \subset E} \bbE (f_{E_0}(X) \, | \, O = E_0) \bbP(O = E_0) \\ 
& = \sum_{E_0 \subset E} \bbE (f(\eta X) \, | \, O = E_0) \bbP(O = E_0) = \bbE( f(\eta X)).
\end{split}
\]
Now comes the place where we need to assume $\theta = 1$: in this case, we know that 
$\eta X$ is a Bernoulli percolation with parameter $(1 - \e{-\beta}) (1 - \delta(\beta, u, 1)$, and therefore we can bound connection probabilities 
for $\eta (1-B)$ by connection probabilities with respect to this Bernoulli percolation. This shows the claim.  
\end{proof}

\section{Proof of Theorem \ref{appearanceblocking}}
			
For this section, we fix $E_0 \subset E$ and write $\tilde \bbP(\cdot) = \tilde \bbP_{\beta,u,\theta}(\cdot) = \bbP_{\beta,u,\theta}( \cdot \, | \, O = E_0)$, which we view as a probability measure on 
$\caC(E_0)$. Note that under $\tilde \bbP$, we have $N_e \geq 1$ almost surely for all $e \in E_0$. We also just write $\delta$ instead of 
$\delta(\beta,u,\theta)$, $E$ instead of $E_0$, and $\caC$ instead of $\caC(E)$ below. 
The strategy of the proof is to show a finite energy property for the percolation measure $(B_e)_{e \in E}$, 
which means that for all $e_0 \in E$ and for all 
$\bar \epsilon \in \{0,1\}^{E\backslash\{e_0\}}$, we have 
\begin{equation} 
\label{eq: uniformDom}
\tilde \bbP(B_{e_0} = 1 \, | \, (B_{e'})_{e' \in E \setminus \{e_0\}} = \bar \epsilon) \geq \delta.
\end{equation}
Then the claim follows by essentially known arguments, which we spell out in Proposition \ref{stochasticdominancecomp} for the convenience of the 
reader. 

The following equality will be used many times below: given $c \in \{O = E\}$, let $c_+(e,j,s)$ be the link configuration where one link of type $s$ is 
inserted into the sequence $c$, at position $j \leq |c|+1$ and on edge $e$. Then 
\begin{equation} \label{c equality}
\tilde \bbP(c_+(e,j,s)) = \tilde \bbP(c) \frac{\beta}{|c|+1} \theta^{|\caL(c_+(e,j,s))| - |\caL(c)|} u^{(1+s)/2} (1-u)^{(1-s)/2}.
\end{equation}
Note that $||\caL(c_+(e,j,s))| - |\caL(c)|| \leq 1$, which gives immediate upper and lower bounds. 

Let us start by a naive try for proving \eqref{eq: uniformDom} which will fail. The uniform domination would hold if we could show 
$\tilde \bbP(B_{e_0}(c) = 1 \, | \, c|_{E \setminus \{e_0\}} = \bar c) \geq \delta$ for all 
$\bar c \in \caC(E \setminus \{e_0\})$, because then we can sum over all $\bar c$ leading to a specific $\bar \epsilon$. 
But this inequality is not true. The reason is that placing 
too many links on an edge adjacent to $e_0$ makes it difficult for $e_0$ to be blocking: to see this, let $E = \{e_0, e'\}$ so that $e_0 \cap e' = \{v\}$. Assume that $\bar c$ is the configuration that has $m$ links on the edge $e'$. For $e_0$ to be blocking, 
we need to place the two links so that none of the links on $e'$ lie between the two links on $e_0$. In other words, if 
$c \in \{N_{e_0} = 1\}$ and the link on $e_0$ is at position $j$, then the additional link that we need to place to get to $\{B_{e_0} = 1\}$ has to be either 
right before or right after that link. Both cases result in the same link configuration. This means that for every such $c \in \{N_{e_0} = 1\}$, there is just one $c \in \{B_{e_0} = 1\}$ that uses the same 
position $j$. As a consequence, \eqref{c equality} gives 
\[
\frac{\tilde \bbP(B_{e_0} = 1 \, | \, c|_{E \setminus \{e_0\}} = \bar c)} {\tilde \bbP(N_{e_0} = 1 \, | \, c|_{E \setminus \{e_0\}} = \bar c)} \leq 
\frac{\hat \theta\beta u^2}{m+1},
\]
which shows that uniform domination must fail since $m$ can be made arbitrarily large.  

So we need to condition on an event coarse enough to prevent this obstruction from happening, but fine enough to be able to estimate conditional 
expectations. For this purpose, for $e\in E$, we introduce 
\[
A_e(k,m)\coloneqq\{f\in E:~ d(e,f)\in[k,m]\},
\]
where $d(e,f)$ is the edge graph distance, i.e., the minimal number of vertices that need to be crossed on a path from the midpoint of $e$ to the midpoint 
of $f$. Also, we set $A(k,m)\coloneqq A_{e_0}(k,m)$. For $\bar c \in \caC(A(2,\infty))$ and $\bar \epsilon \in \{0,1\}^{A(1,2)}$, we define 
\[
\caC_{\bar c, \bar \epsilon} = \{ c \in \caC: c |_{A(2,\infty)} = \bar c, \,  B_e = \bar \eps_{e} \text{ for all } e \in A(1,2) \}.
\]
In words, $\caC_{\bar c, \bar \epsilon}$ consists of sequences that match $\bar c$ for edges with distance two or more from $e_0$, 
and have the correct blocking structure for all edges in $A(1,\infty)$; we only need to demand the latter for edges in $A(1,2)$, since for all other edges it follows from the 
knowledge of $\bar c|_{A(2,\infty)}$. Note that $\caC_{\bar c, \bar \epsilon}$ can be empty in cases where no placement of links on some 
$e \in A(1,1)$ is able to produce the required blocking structure on $A(1,2)$. 

Elements of $\caC_{\bar c, \bar \epsilon}$ can have an arbitrary number of links on the edges of $A(1,1)$, as long as these 
do not interfere with the desired blocking structure. We wish to restrict the number of these links, and for this purpose we 
define a partial order on $\caC$: we say that $c \leq c'$ if $c$ emerges from $c'$ by removing
sequence elements but keeping the relative order of the remaining elements. Then $\caC_{\bar c, \bar \epsilon}$ and $\caC_{\bar c, \bar \epsilon}\cap\{c\in\caC:~B_{e_0}=1\}$ contain minimal elements, and we write 
$\caC_{\bar c, \bar \epsilon}^{\rm min}$ for the union of the two sets of minimal elements. Let us remark already now that for 
$c \in \caC_{\bar c, \bar \epsilon}^{\rm min}$, each edge $\tilde e \in A(1,1)$ has at most $\min\{2, K - 1\}$ links. The reason is that, on the one hand, 
if $\bar \eps_{\tilde e} = 1$, 
there must be precisely two links on $\tilde e$. On the other hand, there are 
at most $K-1$ edges adjacent to $\tilde e$ that are not contained in $A(1,1)$. Each of these edges may be designated as non-blocking, and this may or may not require a link 
on $\tilde e$ to destroy what would otherwise be a blocking structure. Any links beyond these required ones can then be removed. Note that, for any edge $e^\prime\in A(1,1)$ neighbouring $\tilde e$ such that two crosses are placed on $e^\prime$ and $\bar\epsilon_{e^\prime}=0$, there exists a link placed between these two crosses on some other edge neighbouring $e^\prime$: if there would be no such link between the two crosses that are placed on $e^\prime$, we could remove one of the two crosses on $e^\prime$ and still get an element from $\caC_{\bar c, \bar \eps}$ contradicting minimality. Hence, there is no need to place a link on $\tilde e$ to destroy the blocking structure on $e^\prime$ and the number of links in $c\in\caC_{\bar c, \bar \eps}^\text{min}$ on edges from $A(1,1)$ can indeed by bounded by $\min\{2,K-1\}$ since there are at most $K-1$ neighbouring edges to an edge from $A(1,1)$ that are not contained in $A(1,1)$.

For $c \in \caC_{\bar c, \bar \epsilon}^{\rm min}$, set 
\[
\mathrm{Ex}(c) = \{ \tilde c \in \caC_{\bar c, \bar \epsilon} : c \leq \tilde c \}.
\]
As a first step of the proof, we show 

\begin{lemma} \label{c vs Ex(c)}
For each $c \in \caC_{\bar c, \bar \epsilon}^{\rm min}$, we have 
\begin{equation}
\label{Cmin inequality}
\tilde \bbP (\mathrm{Ex}(c)) \leq \e{\beta^+ (2K-2)} \tilde \bbP(\{c\}),
\end{equation}
where 
$
\beta^+\coloneqq\hat\theta\beta$ and $\hat \theta = \max\{\theta, \theta^{-1}\}$.
\end{lemma}

\begin{proof}
Let us re-introduce the notation $\tilde\P_{\beta,u,\theta} = \tilde \bbP$ for now. We start by a known estimate that relates parameters $\theta$ and $1$, see e.g.\ \cite{Bjoernberg2018}. 
For $c \in \caC_{\bar c, \bar \epsilon}^{\rm min}$, we have  
\begin{align*}
	\frac{\tilde\P_{\beta,u,\theta}(\mathrm{Ex}(c))}{\tilde\P_{\beta,u,\theta}(\{c\})}
	&=\sum_{\tilde c\in\mathrm{Ex}(c)}
	\frac{\beta^{|\tilde c|-|c|}|c|!}{|\tilde c|!}\theta^{|\mathcal{L}(\tilde c)|-|\mathcal{L}(c)|}
	u^{\numb{1}(\tilde c)-\numb{1}(c)}(1-u)^{\numb{-1}(\tilde c)-\numb{-1}(c)}\\
	& \leq \sum_{\tilde c\in\mathrm{Ex}(c)}\frac{(\beta\hat\theta)^{|\tilde c|-|c|}|c|!}{|\tilde c|!}
	u^{\numb{1}(\tilde c)-\numb{1}(c)}(1-u)^{\numb{-1}(\tilde c)-\numb{-1}(c)}\\
	&= \frac{\tilde\P_{\beta^+,u,1}(\mathrm{Ex}(c))}{\tilde\P_{\beta^+,u,1}(\{c\})}
\end{align*}
From now on, we will work with $\tilde \bbP_{\beta^+,u,1}$ which we abbreviate as $\tilde \bbP_1$. As discussed above, 
$c \in \caC_{\bar c, \bar \epsilon}^{\rm min}$ can not be extended on edges $e \in A(1,1)$ with $\bar \eps_e = 1$
without leaving $\caC_{\bar c, \bar \epsilon}$. So $\mathrm{Ex}(c)$ arises from $c$ by adding links to the edges where 
$\bar \epsilon_e = 0$. Assume that there are $J$ of such edges. 
We construct elements of ${\rm Ex}(c)$ by stepwise adding crosses and double 
bars to each of the $J$ edges, so in total there will be $2J$ steps. We add crosses in odd steps and double bars in even ones. 
Let $m_i$ be the number of links that we added in the $i$-th step. In the $j$-the step, there are thus 
$|c| + \sum_{k=1}^{j-1} m_k + 1$ positions in the extended sequence in which to place the links. Since they are all of the same type, and 
on the same edge, the order in which we place them is irrelevant, and we get 
\[
\begin{split}
& \frac{\tilde \bbP_1(\mathrm{Ex}(c))}{\tilde \bbP_1(\{c\})}  \\
& = \frac{|c|!}{(\beta^+)^{|c|}} \sum_{m_1, \ldots, m_{2J} = 0}^\infty 
\frac{(\beta^+)^{|c| + \sum_{k=1}^{2J} m_k}}{(|c| + \sum_{k=1}^{2J} m_k)!} \prod_{j=1}^{2J} u^{m_{2j-1}}(1-u)^{m_{2j}} 
\frac{\prod_{k=1}^{2J} (|c| + \sum_{j=1}^{k-1} m_j + 1)^{m_k}}{\prod_{k=1}^{2J} m_k!}
\end{split}
\]
Since 
\[
\frac{\prod_{k=1}^{2J} (|c| + \sum_{j=1}^{k-1} m_j + 1)^{m_k} |c|!}{(|c| + \sum_{k=1}^{2J} m_k)!} \leq 1,
\]
we obtain 
\[
\begin{split}
\frac{\tilde \bbP_1(\mathrm{Ex}(c))}{\tilde \bbP_1(\{c\})} & \leq  \sum_{m_1, \ldots, m_{2J} = 0}^\infty  
\prod_{k=1}^{2J} \frac{(\beta^+)^{m_k}}{m_k!}  \prod_{j=1}^{2J} u^{m_{2j-1}}(1-u)^{m_{2j}}  \\
& = \prod_{k=1}^J \sum_{m_{2k-1}=1}^\infty \frac{(\beta^+ u)^{m_{2k-1}}}{(m_{2k-1})!} \sum_{m_{2k}=1}^\infty \frac{(\beta^+ (1-u))^{m_{2k}}}{(m_{2k})!} 
= \prod_{k=1}^J  \e{\beta^+ u} \e{\beta^+ (1- u)}\\
& = \e{\beta^+ J}.
\end{split}
\]
The result now follows since $J \leq 2K-2$. 
\end{proof}

As an immediate consequence, we note that 
\begin{equation} \label{C Cmin estimate}
	\begin{aligned}
		\tilde \bbP(\caC_{\bar c, \bar \eps}) &= \tilde \bbP(\bigcup_{c \in \caC_{\bar c, \bar \eps}^{\text{min}}} \mathrm{Ex}(c)) \leq 
		\sum_{c \in \caC_{\bar c, \bar \eps}^{\text{min}}} \tilde \bbP(\mathrm{Ex}(c))\leq \e{\beta^+(2K-1)} 
		\sum_{c \in \caC_{\bar c, \bar \eps}^{\text{min}}} \tilde \bbP(\{c\}) \\
		&=  \e{\beta^+(2K-2)} 
		\tilde \bbP(\caC_{\bar c, \bar \eps}^{\text{min}}).
	\end{aligned}
\end{equation}

Next, we show 

\begin{lemma} \label{secondStepLemma}
For all $\bar c \in \caC(A(2,\infty))$ and all $\bar \eps \in \{0,1\}^{A(1,2)}$, we have  
\[
\tilde \bbP( \{B_{e_0} = 1 \} \cap \caC_{\bar c, \bar \eps}^{\rm{min}}) \geq \delta_0
\tilde \bbP(\caC_{\bar c, \bar \eps}^{\rm{min}}) 
\]
with $\delta_0\coloneqq \left(\frac{\hat \theta (8K-4)(2K-1)}{\beta u^2}+1\right)^{-1}$.
\end{lemma}
We start with a combinatorial consideration. For $e_0=\{v_0,w_0\}$, let $c\in \caC^{\rm min}_{\bar c, \bar \eps}|_{A(1,\infty)\setminus E_{{\rm block},v_0}}$,
\begin{align*}
	E_{{\rm block},v_0}&\coloneqq\{e\in E:~v_0\in e,~\bar\eps_{e}=1\},~\text{and}\quad
	Q(c)\coloneqq \{\tilde c\in \caC_{\bar c, \bar \epsilon}^{\rm min}|_{A(1,\infty)}:~c\leq \tilde c\}.
\end{align*}
In words, $Q(c)$ contains the link configurations obtained from $c$ by adding two crosses on each $e \in E_{{\rm block},v_0}$ in such a way that all blocking requirements are satisfied. Also, for $\tilde c \in Q(c)$, 
we have $|\tilde c| = |c| + 2 |  E_{{\rm block},v_0} |$. 
For $\tilde c=(\tilde e_i,\tilde s_i)_{i=1}^n \in Q(c)$ and $e\in E_{{\rm block},v_0}$, 
we let $i_e(\tilde c)$ be the position of the first 'blocking link' on $e$, and $j_e(\tilde c)$ the second position, i.e.,
\[
i_e(\tilde c) = \min\{i\in[|\tilde c|]:~\tilde e_i=e\}, \qquad 
	j_e(\tilde c) = \max\{i\in[|\tilde c|]:~\tilde e_i=e\}.
\]
Write $\nu_c$ for the uniform distribution on $Q(c)$. Then, we have
\begin{lemma}\label{lem-combinatoric}
	For all $c\in \caC^{\rm min}_{\bar c, \bar \eps}|_{A(1,\infty)\setminus E_{{\rm block},v_0}}$,
	\begin{align*}
		\E_{\nu_c}\left(\sum_{e\in E_{{\rm block},v_0}}j_e-i_e\right)\leq \frac{|E_{{\rm block},v_0}|}{2|E_{{\rm block},v_0}|+1}(|c|+2|E_{{\rm block},v_0}|+1).
	\end{align*}
\end{lemma}
\begin{proof} Let us write $n = |E_{{\rm block},v_0}|$. For given $\tilde c \in Q(c)$, we order the numbers $i_e(\tilde c)$ and $j_e(\tilde c)$, 
$e \in E_{{\rm block},v_0}$, by size, and write them as $i_1(\tilde c) < i_2(\tilde c) < \ldots i_n(\tilde c)$, and the same for the $j_k(\tilde c)$. 
Since all $e \in E_{{\rm block},v_0}$ share the vertex $v$, blocking intervals cannot overlap, and we obtain
\[
	0 =: j_0(\tilde c) < i_1(\tilde c) < j_1(\tilde c) < i_2(\tilde c) < \cdots < i_{n}(\tilde c) < j_n(\tilde c) < i_{n+1}(\tilde c) := |c| + 2n +1 = 
	|\tilde c| + 1.
\]
Write $I(\tilde c)=(i_k(\tilde c))_{k\leq n}$ and $J(\tilde c)=(j_k(\tilde c))_{k \leq n}$ for the ordered sequence of all initial, respectively final, blocking links, and write $e_k(\tilde c)$ for the unique element of $E_{{\rm block},v_0}$ on which the links with indices 
$i_k(\tilde c)$ and $j_k(\tilde c)$ are located.

Let us fix a sequence $\mathbf{J}=(\mathbf{j}_k)_{k \leq n}$ with $0 < \mathbf{j}_1 < \ldots \mathbf{j}_n < |c| + 2n + 1$, and an ordering 
$\bse = (\bse_1, \ldots \bse_n)$ of $E_{{\rm block},v_0}$. We set 
\[
Q_{\mathbf{J}, \bse}(c) = \{\tilde c \in Q(c):~ J(\tilde c) = \mathbf{J}, (e_1(\tilde c), \ldots, e_n(\tilde c)) = \bse\}.
\]
This may be empty, for example 
if one of the $\mathbf{j}_k$ appears right after the previous one, or right after a link in 
$c$ that sits on an edge neighbouring $\bse_k$. When the set is not empty, it always contains the element $\tilde c_0 \in Q(c)$ where 
$i_k(\tilde c_0) = j_k(\tilde c_0) - 1$. Indeed we can construct all elements of $Q_{\mathbf{J}, \bse}(c)$ by starting from this configuration 
and swapping the order of the $k$-th initial blocking link with its preceding link until the preceding link is either a link from $c$ on an edge 
neighbouring $e_k(\tilde c_0)$, or we have reached the position $\mathbf{j}_{k-1}$. From these considerations it is clear that given $J=\mathbf{J}$ and 
$(e_1, \ldots e_n) = \bse$, the $k$-th initial blocking link is distributed uniformly (under $\nu_c$) over an interval that is shorter or equal to 
$\mathbf{j}_k - \mathbf{j}_{k-1}$. This implies that 

	\begin{align*}
		\E_{\nu_c}(j_k-i_k|~Q_{\mathbf{J}, \bse}(c)) \leq \E_{\nu_c}(i_k-j_{k-1}|~Q_{\mathbf{J}, \bse}(c))
	\end{align*}
	and hence,
	\begin{align*}
		\E_{\nu_c}(j_k-i_k)\leq \E_{\nu_c}(i_k-j_{k-1}).
	\end{align*}
	With an analogous proof, we can show for all $k \leq n$
	\begin{align*}
		\E_{\nu_c}(j_k-i_k)\leq \E_{\nu_c}(i_{k+1}-j_{k}).
	\end{align*}
	Let $k_0 \leq n$ be a maximizer of $k\mapsto \E_{\nu_c}(i_{k}-j_{k-1})$. Then
	\begin{align*}
		\E_{\nu_c}\left(\sum_{k=1}^{n} j_k-i_k\right)&\leq \sum_{k=1}^{k_0-1}\E_{\nu_c}(i_k-j_{k-1})+\sum_{k=k_0+1}^{n+1}\E_{\nu}(i_k-j_{k-1})\\
		&\leq \frac{n}{n+1} \sum_{k=1}^{n+1}\E_{\nu_c}(i_k-j_{k-1}).
	\end{align*}
	Since $\sum_{k=1}^{n}(j_k-i_k)+\sum_{k=1}^{n+1}(i_k-j_{k-1})=|c|+2n+1$, we conclude
	\begin{align*}
		\E_{\nu_c}\left(\sum_{k=1}^{n}j_k-i_k\right)\leq \frac{n}{2n+1}(|c|+2n+1)
	\end{align*}
	which implies the claim.
\end{proof}
Let $E_{\rm block}\coloneqq \{e\in A(1,1):~\bar\epsilon_e=1\}$. Now, let $c\in \caC_{\bar c, \bar \eps}^{\rm min}|_{A(1,\infty)\backslash E_{\rm block}}$. We define
\begin{align}
	Q_1(c)& \coloneqq \{ \tilde c \in \{N_{e_0} = 1\} \cap  \caC_{\bar c, \bar \eps}^{\text{min}}:~ c \leq \tilde c,~\tilde c|_{\{e_0\}}=(e_0,1)\},\label{eq-Q_1}\\
	Q_2(c)& \coloneqq \{ \tilde c \in \{B_{e_0} = 1\} \cap  \caC_{\bar c, \bar \eps}^{\text{min}}:~ c \leq \tilde c\},\text{ and}\label{eq-Q_2}\\
	\tilde Q(c)&\coloneqq Q_1(c)|_{A(1,\infty)}=Q_2(c)|_{A(1,\infty)}.\notag
\end{align}
Moreover, let $\mu_c$ denote the uniform measure on $\tilde Q(c)$. Given a realisation $\tilde c=(\tilde e_i,\tilde s_i)_{i=1}^{\tilde n}$ of $\mu_c$, we want to add links to $\tilde c$ to get configurations from $Q_1(c)$ and $Q_2(c)$. We say that we add one cross at $k\in[\tilde n]$ or two crosses at $k,l\in[\tilde n]$ with $k<l$ to $\tilde c$ on $e_0$ if we construct the link configuration $c^\prime=(e^\prime_i,s^\prime_i)_{i=1}^{n^\prime}$ with $c^\prime|_{A(1,\infty)}=\tilde c$ and, for adding one cross, $e^\prime_k=e_0$ and $c^\prime|_{\{e_0\}}=((e_0,1))$ or respectively, for adding two crosses, $e^\prime_k=e^\prime_{l+1}=e_0$ and $c^\prime|_{\{e_0\}}=((e_0,1),(e_0,1))$. Then, there exist disjoint intervals $((a_m,b_m])_{m=1}^{m_0}$ with $a_m,b_m\in \N$ for all $m\in [m_0]$ such that adding two crosses to $\tilde c$ yields an element from $Q_2(c)$ if and only if the crosses are added at $k,l\in(a_m,b_m]\cap \N$ for some $m\in [m_0]$. Also, adding one cross yields an element from $Q_1(c)$ if and only if the cross is added at $k\in \N\cap \bigcup_{m\in[m_0]}(a_m,b_m]$. For all $m\in[m_0]$, we set $L_m\coloneqq b_m-a_m$. Note that we can consider $(a_m,b_m)_{m\in[m_0]}$ and $(L_m)_{m\in[m_0]}$ to be random variables w.r.t. $\mu_c$.
\begin{lemma}\label{lem-estimatefreespots}
	For all $c\in \caC_{\bar c, \bar \eps}^{\rm min}|_{A(1,\infty)\backslash E_{\rm block}}$,
	\begin{align*}
		\E_{\mu_c}\left(\sum_{m\in [m_0]}L_m\right)\geq \frac{|c|+1+|E_{\rm block}|}{|E_{\rm block}|+1}.
	\end{align*}
\end{lemma} 
\begin{proof}
	Let $\tilde c=(\tilde e_i,\tilde s_i)_{i=1}^{\tilde n}$ be a realisation of $\mu_c$. We note that $k\in\bigcup_{m\in [m_0]}(a_m(\tilde c),b_m(\tilde c)]\cap \N$ if and only if there does not exist some $e\in E_{{\rm block}}$ with $e=\tilde e_i=\tilde e_j$ and $i<k<j$. This gives
	\begin{align*}
		\sum_{m\in [m_0]}L_m(\tilde c)\geq |c|+2|E_{{\rm block}}|+1-\sum_{e\in E_{{\rm block}}}(j_e-i_e).
	\end{align*}
	Since all link configurations have the same probability under $\mu_c$, we can find some measure $\tilde \mu$ on $\tilde Q(c)|_{A(1,\infty)\setminus E_{{\rm block},v_0}}$ such that we get $\mu_c$ by first drawing some $\tilde c$ w.r.t. $\tilde \mu$ and then drawing a link configuration with $\nu_{\tilde c}$ from $Q(\tilde c)$. This works also when replacing $v_0$ with $w_0$ and hence, we get with Lemma \ref{lem-combinatoric}
	\begin{align*}
		\E_{\mu_c}\left(\sum_{m\in [m_0]}L_m\right)&\geq |c|+2|E_{{\rm block}}|+1\\
		&-\left(\frac{|E_{{\rm block},v_0}|}{2|E_{{\rm block},v_0}|+1}+\frac{|E_{{\rm block},w_0}|}{2|E_{{\rm block},w_0}|+1}\right)(|c|+2|E_{\rm block}|+1).
	\end{align*}
	Since, for $K>0$, on $\{(x,y)\in[0,\infty)^2:~x+y=K\}$, $$(x,y)\mapsto \frac{x}{2x+1}+\frac{y}{2y+1}$$ attains its maximum at $\left(K/2,K/2\right)$, we conclude the claim.
\end{proof}
\begin{proof}[Proof of Lemma \ref{secondStepLemma}]
Below, we will show the following equality and inequality for all $\bar c \in \caC(A(2,\infty))$ and $\bar \eps \in \{0,1\}^{A(1,2)}$ 
such that $\caC_{\bar c, \bar \eps} \neq \emptyset$: 
\begin{align}
\tilde \bbP((\{N_{e_0} = 1 \}\cup\{B_{e_0}=1\} )\cap \caC_{\bar c, \bar \eps}^{\text{min}} ) & =  \bbP(\caC_{\bar c, \bar \eps}^{\rm{min}}),  \label{eq1} \\
\tilde \bbP(\{ B_{e_0} = 1 \} \cap \caC_{\bar c, \bar \eps}^{\text{min}} ) & \geq d_0
\bbP(\{N_{e_0} = 1\} \cap \caC_{\bar c, \bar \eps}^{\text{min}} ) \label{ineq2} 
\end{align}
for $d_0\coloneqq \frac{\beta u^2}{\hat \theta (8K-4)(2K-1)}$. Once we have this, we obtain the claim with $\delta_0 = \frac{d_0}{1+d_0}$.

Equality \eqref{eq1} follows from an argument that we have already given in front of Lemma \ref{c vs Ex(c)} but we repeat here in the specific setting for the convenience of the reader: suppose there are at least two links placed on $e_0$ in $c\in \caC_{\bar c, \bar \eps}^{\text{min}}$ and $B_{e_0}(c)=0$. This only happens if removing either of these links yields a link configuration not contained in $\caC_{\bar c, \bar \eps}$ anymore. This again is only possible if there is some edge $e^\prime\in A(1,1)$ such that $c$ places two crosses on $e^\prime$ without any link on a neighbouring edge that is not $e_0$, being placed in-between the two crosses and $\bar\epsilon_{e^\prime}=0$. In this case, $c$ is not minimal since one of the two crosses on $e^\prime$ can be removed. Consequently, $c\in \caC^\text{min}_{\bar c, \bar \eps}$ and $B_{e_0}(c)=0$ already implies $N_{e_0}(c)=1$.\\
To show \eqref{ineq2}, we fix $c \in \caC_{\bar c, \bar \eps}^{\text{min}}|_{A(1,\infty)\backslash E_{\rm block}}$, i.e., all the links except the ones on $\{e_0\}\cup E_{\rm block}$ are fixed.
We write 
\[
q(c) = \frac{\beta^{|c|}}{|c|!} u^{\frac12 \sum_{j=1}^n (1 + s_j)}(1-u)^{\frac12 \sum_{j=1}^n (1 - s_j) } \theta^{|\caL(c)|} 
\]
for the weight of $c$ (but note that $\tilde \bbP(c) = 0$ since the condition of at least one link on $\{e_0\}\cup E_{\rm block}$ is not met). We remind the reader of the definitions \eqref{eq-Q_1} and \eqref{eq-Q_2}.\\ 
Since adding two links that form a blocking structure never changes the number of loops in the system, we have for all $\tilde c \in Q_2(c)$ that
\begin{align*}
	\tilde \bbP(\tilde c) &= \frac{\beta^{|c|+2+2|E_{\rm block}|}}{(|c|+2+2|E_{\rm block}|)!} u^{2+|E_{\rm block}|}u^{\frac12 \sum_{j=1}^n (1 + s_j)}(1-u)^{\frac12 \sum_{j=1}^n (1 - s_j) } \theta^{|\caL(c)|}\\
	 &= \frac{(u \beta)^{2+2|E_{\rm block}|}|c|!}{(|c|+2+2|E_{\rm block}|)!} q(c).	
\end{align*}
On the other hand, adding one link to $c$ on the edge $e_0$ changes the number of loops by exactly one, so that for all $\tilde c \in Q_1(c)$, we have 
\begin{align*}
	\tilde \bbP(\tilde c) + \tilde \bbP(\tilde c^\dagger) &\leq \frac{\beta^{|c|+1+2|E_{\rm block}|}}{(|c|+1+2|E_{\rm block}|)!}u^{2|E_{\rm block}|}u^{\frac12 \sum_{j=1}^n (1 + s_j)}(1-u)^{\frac12 \sum_{j=1}^n (1 - s_j) } \theta^{|\caL(c)|} \hat \theta\\
	& = 
	\frac{\beta u^{2|E_{\rm block}|} \hat \theta|c|!}{(|c|+1+2|E_{\rm block}|)!} q(c) 	
\end{align*}
where $\tilde c^\dagger$ is the link configuration where the link on $e_0$ is replaced by the opposite type compared to $\tilde c$. By summation over all $c \in \caC_{\bar c, \bar \eps}^{\text{min}}|_{A(1,\infty)\backslash E_{\rm block}}$, that is, 
\begin{align*}
	\tilde \bbP(\{ B_{e_0} = 1 \} \cap \caC_{\bar c, \bar \eps}^{\text{min}} )&=\sum_{c\in \caC_{\bar c, \bar \eps}^{\rm min}|_{A(1,\infty)\setminus E_{\rm block}}}\sum_{\tilde c\in Q_2(c)}\tilde \P(\{\tilde c\}),\text{ and}\\
	\tilde \bbP(\{ N_{e_0} = 1 \} \cap \caC_{\bar c, \bar \eps}^{\text{min}} )&=\sum_{c\in \caC_{\bar c, \bar \eps}^{\rm min}|_{A(1,\infty)\setminus E_{\rm block}}}\sum_{\tilde c\in Q_1(c)}(\tilde \P(\{\tilde c\})+\tilde \P(\{\tilde c^\dagger\})),
\end{align*}
 we can conclude \eqref{ineq2}, once we have established 
\begin{align}\label{combinatorics}
	\frac{\beta u^2|Q_2(c)|}{|c|+2+2|E_{\rm block}|}\geq d_0\hat \theta|Q_1(c)|.
\end{align}
We remind the reader of the definition of $(L_m)_{m=1}^{m_0}$ and of the way elements from $Q_1(c)$ and $Q_2(c)$ can be constructed starting with an element from $\tilde Q(c)$.\\
From now on, taking again $c\in \caC_{\bar c, \bar \eps}^\text{min}|_{A(1,\infty)\backslash E_{\rm block}}$, we denote by $\mu_c$ the uniform measure on $\tilde Q(c)$ as before. Then we get
\begin{align*}
	|Q_1(c)|&=|\tilde Q(c)|\E_{\mu_c}\left(\sum_{m=1}^{m_0}\frac{L_m+1}{2}L_m\right),\text{ and}\quad
	|Q_2(c)|=|\tilde Q(c)|\E_{\mu_c}\left(\sum_{m=1}^{m_0}L_m\right).
\end{align*}
Noting that, for all $n\in\N$ and $K_0\in \R$, $(x_i)_{i\in[n]_0}\mapsto \sum_{i=0}^n\frac{x_i^2}{2}$ on $\sum_{i=0}^nx_i=K_0$ attains its minimum at $(x_i)_{i\in[n]_0}=\frac{K_0}{n+1}(1)_{i\in[n]_0}$ and that, for $\tilde c\in Q_1(c)|_{A(1,\infty)}$, $m_0\leq 2K-2$, we get
\begin{align*}
	\E_{\mu_c}\left(\sum_{m=1}^{m_0}\frac{L_m+1}{2}L_m\right)&\geq \frac 1 2\E_{\mu_c}\left(m_0\left(\frac{\sum_{m=1}^{m_0}L_m}{m_0}\right)^2\right)
	\geq \frac{1}{4K-2}\E_{\mu_c}\left(\sum_{m=1}^{m_0}L_m\right)^2
\end{align*}
where we have used Jensen's inequality. Hence, \eqref{combinatorics} follows by
\begin{align*}
	\E_{\mu_c}\left(\sum_{m=1}^{m_0}L_m\right)\geq (4K-2)\frac{|c|+2+2|E_{\rm block}|}{\beta u^2}\hat\theta d_0
\end{align*}
as a consequence of Lemma \ref{lem-estimatefreespots}.
\end{proof}
\begin{proof}[Proof of Theorem \ref{appearanceblocking}]
	We need to show \eqref{eq: uniformDom}. We use Lemmata \ref{c vs Ex(c)} and \ref{secondStepLemma} to get for all $\bar c \in \caC(A(2,\infty))$ and $\bar \epsilon \in \{0,1\}^{A(1,2)}$
	\begin{align*}
		\tilde \P(\{B_{e_0}=1\}\cap \caC_{\bar c, \bar \eps})\geq \tilde \P(\{B_{e_0}=1\}\cap \caC^\text{min}_{\bar c, \bar \eps})\geq \delta_0\tilde \P( \caC^\text{min}_{\bar c, \bar \eps})\geq \delta_0\e{-\beta^+(2K-2)}\tilde\P(\caC_{\bar c, \bar \eps})
	\end{align*}
	where we have summed inequality \eqref{Cmin inequality} over all $c\in\caC^\text{min}_{\bar c, \bar \eps}$. Let $\tilde\epsilon\in\{0,1\}^{E\setminus\{e_0\}}$ be arbitrary and set $\bar \epsilon\coloneqq \tilde\epsilon|_{A(1,2)}$. By summing the estimate over all the $\bar c$ such that, for all $e^\prime\in A(3,\infty)$, we have $B_{e^\prime}(\bar c)=\tilde\epsilon_{e^\prime}$, we conclude
	\begin{align*}
			\tilde \P(B_{e_0}=1,~(B_{e'})_{e' \in E \setminus \{e\}} = \tilde\epsilon)\geq \delta_0\e{-\beta^+(2K-2)}\tilde \P((B_{e'})_{e' \in E \setminus \{e\}} = \tilde\epsilon)
	\end{align*}	
	which yields the claim.
\end{proof}
			\appendix
			\section{Stochastic domination}
			This result provides stochastic domination in the case of local domination making it possible to find a coupling between the locally dominating and dominated measure. We will consider measures on $\{0,1\}^I$ for some set $I$ at most countable. For simplification, we identify $I$ with $[N]$ for some $N\in\N\cup\{\infty\}$ with $[\infty]\coloneqq \N$. A similar statement with a slightly different assumption to the one of the following proposition can be found in \cite{Liggett1997}.
			\begin{proposition}\label{stochasticdominancecomp}
				Let $X$ and $Y$ be $\{0,1\}^I$-valued a random variables on such that for every finite set $J\subseteq I$, all $(\epsilon_j)_{j\in J}\subseteq\{0,1\}$ and every $i\in I\backslash J$, we have
				\begin{align*}
					\P\left(X_i=1\Bigg|\forall j\in J:~X_j=\epsilon_j\right)\geq \P\left(Y_i=1\Bigg|\forall j\in J:~Y_j=\epsilon_j\right).
				\end{align*}
				Then $Y$ is stochastically dominated by $X$.
			\end{proposition}
			\begin{proof}
				We define functions $(m_k)_{k\in\N}$ by
				\begin{align*}
					m_k((\epsilon_j,\tilde\epsilon_j)_{j\in[k]})\coloneqq\begin{cases*}
						0 \text{ if }(\epsilon_k,\tilde\epsilon_k)=(0,1),\\
						\P\left(X_k=1\big|\forall j\leq k-1:~X_j=\epsilon_j\right)-\P\left(Y_k=1\big|\forall j\leq k-1:~Y_j=\epsilon_j\right) \\
						~~~~~~~~~~~~~~~~~~~~~~~~~~~~~~~~~~~~~~~~\text{ if }(\epsilon_k,\tilde\epsilon_k)=(1,0),\\
						\P\left(Y_k=1\big|\forall j\leq k-1:~Y_j=\epsilon_j\right) \text{ if }(\epsilon_k,\tilde\epsilon_k)=(1,1),\\
						1-\P\left(Y_k=1\big|\forall j\leq k-1:~Y_j=\epsilon_j\right) \text{ if }(\epsilon_k,\tilde\epsilon_k)=(0,0).
					\end{cases*}
				\end{align*}
				Using Kolmogorovs extension theorem, we find a random variable $Z$ on $\{0,1\}^I\times \{0,1\}^I$ such that for all $k\in\N$ and $(\epsilon_j,\tilde\epsilon_j)_{j\in[k]}$
				\begin{align*}
					\P(Z=(\epsilon_j,\tilde\epsilon_j)_{j\in[k]})=\prod_{l=1}^km_{l}((\epsilon_j,\tilde\epsilon_j)_{j\in[l]}).
				\end{align*}
				A straightforward calculation shows $Z(\cdot,\{0,1\}^I)\overset{d}{=}X(\cdot)$ and $Z(\{0,1\}^I,\cdot)\overset{d}{=}Y(\cdot)$. By the definition of $Z$, we have $Z\in\{(a_i,b_i)_{i\in I}\in\{0,1\}^I:~a_i\geq b_i\}$ $\P$-a.s.. This implies that $Y$ is stochastically dominated by $X$.
			\end{proof}
			
				\subsection*{Acknowledgements}
			Andreas Klippel thanks the Casanuswerk for supporting his PhD studies. 
			Mino Nicola Kraft thanks the Studienstiftung for supporting his PhD studies.

			\nocite{*}
			\bibliographystyle{alpha}
			\bibliography{bibliography (1)}
		\end{document}